\documentclass[12pt, reqno]{amsart}
\usepackage{amssymb, mathrsfs}
\usepackage{latexsym}
\usepackage{amsmath}
\usepackage{amsthm}
\usepackage{amsfonts}
\usepackage{dsfont, upgreek}
\usepackage{mathtools}
\usepackage{epsfig}
\usepackage[nocompress]{cite}
\usepackage{amscd}
\usepackage{graphicx}
\usepackage{tikz-cd}
\usepackage{enumitem}
\setlist[enumerate]{itemsep=0.5ex}

\usepackage{color}
\usepackage{tikz}
\usetikzlibrary{patterns}

\usepackage{geometry}
\usepackage{tikz}
\usetikzlibrary{decorations.markings}
\usetikzlibrary{arrows.meta}

\usepackage{extarrows}

\usepackage{adjustbox}
\usepackage{pgfplots}
\usepgfplotslibrary{colormaps}
\usepackage{slashed}

\usepackage[colorlinks=true,linkcolor=blue,citecolor=blue,urlcolor=blue]{hyperref}
\providecommand{\MR}[1]{}
\renewcommand{\MR}[1]{\href{https://mathscinet.ams.org/mathscinet-getitem?mr=#1}{MR #1}}

\usepackage[colorinlistoftodos,prependcaption,textsize=tiny]{todonotes}  

\usepackage[all,cmtip]{xy}
\theoremstyle{plain}
\newtheorem{theorem}{Theorem}[section]
\newtheorem{proposition}[theorem]{Proposition}
\newtheorem{lemma}[theorem]{Lemma}
\newtheorem{corollary}[theorem]{Corollary}

\theoremstyle{definition}

\newtheorem*{claim*}{Claim}

\theoremstyle{remark}

\numberwithin{equation}{section}

\newcommand{\Sc}{\mathrm{Sc}}

\newcommand{\divv}{\mathrm{div}}

\newcommand{\Ric}{\mathrm{Ric}}

\newcommand{\interior}[1]{%
	{\kern0pt#1}^{\mathrm{\,o}}%
}

\makeatletter
\let\save@mathaccent\mathaccent
\newcommand*\if@single[3]{%
	\setbox0\hbox{${\mathaccent"0362{#1}}^H$}%
	\setbox2\hbox{${\mathaccent"0362{\kern0pt#1}}^H$}%
	\ifdim\ht0=\ht2 #3\else #2\fi
}
\newcommand*\rel@kern[1]{\kern#1\dimexpr\macc@kerna}
\newcommand*\overbar[1]{\@ifnextchar^{{\wide@bar{#1}{0}}}{\wide@bar{#1}{1}}}
\newcommand*\wide@bar[2]{\if@single{#1}{\wide@bar@{#1}{#2}{1}}{\wide@bar@{#1}{#2}{2}}}
\newcommand*\wide@bar@[3]{%
	\begingroup
	\def\mathaccent##1##2{%
		\let\mathaccent\save@mathaccent
		\if#32 \let\macc@nucleus\first@char \fi
		\setbox\z@\hbox{$\macc@style{\macc@nucleus}_{}$}%
		\setbox\tw@\hbox{$\macc@style{\macc@nucleus}{}_{}$}%
		\dimen@\wd\tw@
		\advance\dimen@-\wd\z@
		\divide\dimen@ 3
		\@tempdima\wd\tw@
		\advance\@tempdima-\scriptspace
		\divide\@tempdima 10
		\advance\dimen@-\@tempdima
		\ifdim\dimen@>\z@ \dimen@0pt\fi
		\rel@kern{0.6}\kern-\dimen@
		\if#31
		\overline{\rel@kern{-0.6}\kern\dimen@\macc@nucleus\rel@kern{0.4}\kern\dimen@}%
		\advance\dimen@0.4\dimexpr\macc@kerna
		\let\final@kern#2%
		\ifdim\dimen@<\z@ \let\final@kern1\fi
		\if\final@kern1 \kern-\dimen@\fi
		\else
		\overline{\rel@kern{-0.6}\kern\dimen@#1}%
		\fi
	}%
	\macc@depth\@ne
	\let\math@bgroup\@empty \let\math@egroup\macc@set@skewchar
	\mathsurround\z@ \frozen@everymath{\mathgroup\macc@group\relax}%
	\macc@set@skewchar\relax
	\let\mathaccentV\macc@nested@a
	\if#31
	\macc@nested@a\relax111{#1}%
	\else
	\def\gobble@till@marker##1\endmarker{}%
	\futurelet\first@char\gobble@till@marker#1\endmarker
	\ifcat\noexpand\first@char A\else
	\def\first@char{}%
	\fi
	\macc@nested@a\relax111{\first@char}%
	\fi
	\endgroup
}
\makeatother
\begin{document}

\title{Scalar Curvature, Sharp Bottom Spectrum and Geometric Rigidity}

\author{Jinmin Wang}
\address[Jinmin Wang]{Institute of Mathematics, Chinese Academy of Sciences}
\email{jinmin@amss.ac.cn}
\thanks{The first author is partially supported by NSFC 12501169.}
\author{Bo Zhu}
\address[Bo Zhu]{Yau Mathematical Sciences Center, Tsinghua University}
\email{zhub@tsinghua.edu.cn}
\thanks{The second author is  partially supported by NSFC 12501066.}

\begin{abstract}
We prove rigidity in the equality case of the sharp bottom spectrum estimate under scalar curvature lower bound. Under the same topological assumptions as in our previous work, a closed manifold $(M,g)$ with $\Sc_g\geq -n(n-1)$ and
$\lambda_1(\widetilde M,\widetilde g)=(n-1)^2/4$  must be hyperbolic. This gives rigidity results for closed hyperbolic manifolds and for closed manifolds admitting a metric of nonpositive sectional curvature.
\end{abstract}
\maketitle

\section{Introduction}

Let $(X^n,g)$ be a connected, complete, noncompact Riemannian manifold, and
let $\Delta$ denote its Laplace--Beltrami operator. The $L^2$ bottom of the
spectrum of $\Delta$ is given by the variational characterization
(see \cite[Chapter~6, Definition~6.3]{Li_geometric_analysis})
\begin{equation*}
    \lambda_1(X,g):=\inf\left\{\frac{\int_X|\nabla f|^2}{\int_X f^2}: f\in C_c^\infty(X),\ f\not\equiv 0\right\}.
\end{equation*}

This quantity is a fundamental spectral invariant of the large-scale
geometry of $(X,g)$. A classical comparison theorem of S.Y. Cheng asserts
that a Ricci curvature lower bound gives a sharp upper bound for
$\lambda_1$ and proves if $\operatorname{Ric}_g\geq -(n-1)g$, then
\begin{equation}\label{eq:cheng-spectrum}
    \lambda_1(X,g)\leq \frac{(n-1)^2}{4}.
\end{equation}
See \cite[Theorem~4.2]{Cheng_eigenvalue}. The value
$(n-1)^2/4$ is attained by hyperbolic space $\mathbb H^n$, and hence
\eqref{eq:cheng-spectrum} is sharp. The equality and near-equality cases
have been extensively studied in connection with rigidity phenomena.
In particular, Li--Wang developed harmonic function methods to obtain
splitting and rigidity results for complete noncompact manifolds satisfying
$\operatorname{Ric}_g\geq -(n-1)g$
(see \cite{Li_Wang_positive_spectrum_1,Li_Wang_positive_spectrum_2}).

The scalar curvature analogue of Cheng's estimate is subtler, since scalar
curvature does not directly control the Laplacian by comparison geometry. In
dimension three, Munteanu and J. Wang proved the sharp bottom
spectrum estimate under a negative scalar curvature lower bound, using
harmonic function techniques (see \cite[Theorem~1.1]{MunteanuWang}). In our previous
work, we established a higher-dimensional version for closed manifolds whose
universal cover is spin, using Dirac operators and higher index theory
(see \cite[Theorem~1.2]{WangZhu24}). We recall the result in the form
needed below.

\begin{theorem}[Wang--Zhu, see {\cite[Theorem~1.2]{WangZhu24}}]\label{thm: cocompact_rigidity}
Let $(M^n,g)$ be a closed Riemannian manifold with fundamental group
$\Gamma$, and let $(\widetilde M,\widetilde g)$ be its Riemannian universal
cover. Suppose that $\Sc_g\geq -\kappa$ for some constant $\kappa\geq0$.
If
\begin{itemize}
\item $M$ is rationally essential, namely the fundamental class $[M]$ is
non-zero in $H_*(B\Gamma,\mathbb Q)$,
\item $\widetilde M$ is spin, and $\Gamma$ satisfies the Strong Novikov
Conjecture.
\end{itemize}
Then
\[
    \lambda_1(\widetilde{M}, \widetilde{g}) \leq \frac{n-1}{4n} \kappa.
\]
Moreover, if equality holds in this estimate, then $\Sc_g\equiv -\kappa$.
\end{theorem}

The case $\kappa=0$ gives $\lambda_1(\widetilde M,\widetilde g)=0$, and the
equality case in our previous work implies that $(M,g)$ is Ricci-flat and indeed flat.
See Theorem~1.2 and the subsequent remark in \cite{WangZhu24}. The remaining
case is therefore $\kappa>0$. Since the estimate is invariant under scaling
after the corresponding normalization of the scalar curvature lower bound,
we shall henceforth use the hyperbolic normalization
$\kappa=n(n-1)$.

\medskip

We now turn to the rigidity problem for the equality case of the sharp
spectral bound. Under the Ricci curvature lower bound
$\Ric_g\geq -(n-1)g$, X. Wang gave a rigidity characterization of
hyperbolic space from the sharp bottom spectrum via a sharp estimate for the
Kaimanovich entropy (see \cite[Theorem~1.4]{Wang_eigenvalue}). In the scalar
curvature setting, Munteanu and J. Wang proved the corresponding
three-dimensional rigidity theorem by harmonic function methods
(see \cite[Theorem~1.3]{MunteanuWang}).

The scalar curvature setting is substantially more delicate than the Ricci
curvature setting. A lower bound for scalar curvature gives no direct
Laplacian comparison, entropy comparison, or control of the Ricci curvature.
The sharp estimate above yields scalar curvature rigidity in the equality
case, namely
\(
\Sc_g \equiv -n(n-1).
\)

\medskip

In this paper, we upgrade this scalar curvature rigidity
to full geometric rigidity, by using the harmonic spinor to recover the Ricci
information that is otherwise absent in the equality case.

\begin{theorem}\label{thm:rigidity}
Let $(M^n,g)$ be a closed Riemannian manifold with fundamental group
$\Gamma$, and let $(\widetilde M,\widetilde g)$ be its Riemannian universal
cover. Suppose that $\Sc_g\geq -n(n-1)$. If
\begin{itemize}
\item $M$ is rationally essential,
\item $\widetilde M$ is spin, and $\Gamma$ satisfies the Strong Novikov
Conjecture,
\item the bottom spectrum of the universal cover is sharp, namely
    \[
    \lambda_1(\widetilde M,\widetilde g)=\frac{(n-1)^2}{4}.
    \]
\end{itemize}
Then $(M,g)$ is hyperbolic.
\end{theorem}

As an immediate corollary, we prove the following rigidity theorem for
metrics on a closed hyperbolic manifold.

\begin{corollary}[Hyperbolic case]\label{cor:hyperbolic-rigidity}
Let $M^n$ be a closed hyperbolic manifold, and let $g$ be a Riemannian
metric on $M$ with $\Sc_g\geq -n(n-1)$. If
\[
    \lambda_1(\widetilde M,\widetilde g)=\frac{(n-1)^2}{4},
\]
then $(M,g)$ is hyperbolic.
\end{corollary}

This corollary should be viewed alongside two classical large-scale
invariants of the universal cover, the Cheeger constant
$h(\widetilde M,\widetilde g)$ and the volume entropy
$\operatorname{Ent}(\widetilde M,\widetilde g)$. They are related to the
bottom spectrum by the standard estimates
\[
    \frac{h(\widetilde M,\widetilde g)^2}{4}
    \leq \lambda_1(\widetilde M,\widetilde g)
    \leq \frac{\operatorname{Ent}(\widetilde M,\widetilde g)^2}{4};
\]
in the hyperbolic case, $h=\operatorname{Ent}=n-1$ and
$\lambda_1=(n-1)^2/4$. Under the stronger
Ricci lower bound $\Ric_g\geq -(n-1)g$, the usual comparison theorems give
$\operatorname{Ent}(\widetilde M,\widetilde g)\leq n-1$ and
$\lambda_1(\widetilde M,\widetilde g)\leq (n-1)^2/4$, and the equality case
for volume entropy was proved by Ledrappier--X. Wang using an integral
formula for volume entropy (see \cite{Fran_Wang_entropy_rigidity}); Liu
later gave a short proof of this rigidity result (see
\cite{Liu_volume_entropy}). However,
Kazaras--Song--Xu showed that the scalar curvature lower bound alone does not
force the sharp entropy bound and every closed hyperbolic three-manifold admits a
metric with $\Sc_g\geq -6$ but volume entropy larger than $2$
(see \cite{kazaras2024scalarcurvaturevolumeentropy}). Corollary
\ref{cor:hyperbolic-rigidity} shows that the sharp bottom spectrum is more
rigid in this scalar-curvature setting, preserving the hyperbolic value of
$\lambda_1$ already forces the metric to be hyperbolic.

We also have the following geometric version. The essential topological
assumption in Theorem~\ref{thm:rigidity} is the Strong Novikov Conjecture for
the fundamental group $\Gamma$, and the existence of a nonpositively curved metric
is one geometric condition ensuring this assumption. Thus the same sharp spectral condition
becomes a pure rigidity criterion for the metric $g$.

\begin{corollary}
\label{cor:npc-rigidity}
Let $M^n$ be a closed manifold admitting a Riemannian metric of
nonpositive sectional curvature, and let $g$ be a Riemannian metric on $M$
with $\Sc_g\geq -n(n-1)$. If
\[
    \lambda_1(\widetilde M,\widetilde g)=\frac{(n-1)^2}{4},
\]
then $(M,g)$ is hyperbolic.
\end{corollary}

\medskip
\subsection*{Difficulties of the main theorem}

The main difficulty is that the assumption is a lower bound for scalar
curvature rather than for Ricci curvature. In the scalar curvature setting,
the classical techniques, volume comparison and harmonic function techniques based on Ricci curvature, are unavailable, and
X. Wang's rigidity theorem cannot be applied directly. The key point of the
proof is therefore to recover the missing Ricci curvature information from the
equality case of the sharp spinorial estimate. The new step is to
turn scalar curvature rigidity into full geometric rigidity.

Recall that the sharp bottom spectrum estimate in the authors' previous work is based on almost-harmonic
spinors on the universal cover (see
Theorem~\ref{thm: cocompact_rigidity}). In the equality case, these spinors
need not have a global limit on the noncompact cover, and their mass may
escape to infinity. Lemma~\ref{lemma:limit} overcomes this by localizing the
almost-equality information and recentering the spinors by suitable deck
transformations. This produces a nontrivial local limiting spinor satisfying
the equality conditions in Propositions~\ref{prop:refinedKato} and
\ref{prop:Lichnerowicz}.

These equality conditions are the new rigidity mechanism that yields rigid local
first-order spinorial equations. These local equations then
convert into a pointwise relation involving Ricci curvature. Combining
this relation with the scalar curvature equality and the Laplacian equation
for the norm of the local spinor, we show that the traceless Ricci tensor vanishes. Thus the
proof upgrades the scalar curvature equality to the Einstein condition
$\Ric_g=-(n-1)g$, after which X. Wang's sharp bottom-spectrum rigidity theorem
implies Theorem~\ref{thm:rigidity}.

\subsection*{Organization of the paper}
Section~\ref{sec:ingredients}, we recall the analytic input from the sharp
bottom spectrum estimate, including the nonvanishing higher index, the
refined Kato inequality, and the equality statement for the scalar curvature.
Section~\ref{sec:proof-rigidity}, we first prove the localization argument,
constructs the strongly constrained local spinor arising in the equality
case, and then translates this spinorial information into geometry to
complete the proof of Theorem~\ref{thm:rigidity}.

\medskip

\textbf{Acknowledgments.} We thank S. T. Yau, J. P. Wang, and X. D. Wang for helpful discussions and comments.

\section{Some essential ingredients}\label{sec:ingredients}
In this section, we recall the main analytic ingredients from the authors'
previous work (see \cite[Theorem~1.2]{WangZhu24}) that will be used in the
rigidity argument. The sharp bottom spectrum theorem has two parts, the
estimate
\[
    \lambda_1(\widetilde M,\widetilde g)
    \leq \frac{n-1}{4n}\kappa
\]
under the scalar curvature lower bound $\Sc_g\geq -\kappa$, and the equality
statement that equality forces $\Sc_g\equiv -\kappa$.

The proof starts from the higher index of the Dirac operator. Under the
topological assumptions in Theorem~\ref{thm: cocompact_rigidity}, this index
is nonzero. Hence the lifted Dirac operator on the universal cover cannot be
invertible, so zero lies in its spectrum. This gives the sequence of
almost-harmonic spinors used as test sections in the spectral estimate.
\begin{proposition}\label{prop:zeroInSpec}
Under the topological assumptions in Theorem~\ref{thm: cocompact_rigidity},
zero lies in the spectrum of the lifted Dirac operator $\widetilde D$.
\end{proposition} 

The second ingredient is a quantitatively refined Kato inequality. It converts
control of the spinor into control of its length function, which is the bridge
from the Dirac operator to the sharp bottom spectrum of the Laplacian. We give
a proof for completeness.

\begin{proposition}\label{prop:refinedKato}
There exists a dimensional constant $c_n>0$ such that for any smooth spinor
$\sigma$, one has
\[
    \frac{n}{n-1}|\nabla|\sigma||^2
    \leq |\nabla\sigma|^2
    +c_n|\widetilde D\sigma|^2+c_n|\widetilde D\sigma|\,|\nabla\sigma|.
\]
\end{proposition}
\begin{proof}
Let $\{e_i\}_{i=1}^n$ be a local orthonormal frame and set
$\sigma_i=\nabla_{e_i}\sigma$. At a point where $\sigma\neq 0$,
\[
    |\sigma|^2|\nabla|\sigma||^2
    =\sum_{i=1}^n\bigl(\operatorname{Re}\langle\sigma_i,\sigma\rangle\bigr)^2.
\]
Define the Penrose operator by
\[
    \mathcal P_i\sigma
    :=\nabla_{e_i}\sigma+\frac1n c(e_i)\widetilde D\sigma .
\]
By the Clifford relation, $\sum_i c(e_i)\mathcal P_i\sigma=0$. The same
relation also gives the norm identity
\begin{equation}\label{eq:penrose-norm-identity}
    \sum_{i=1}^n|\mathcal P_i\sigma|^2
    =|\nabla\sigma|^2-\frac1n|D\sigma|^2 .
\end{equation}
Moreover, the algebraic refined Kato inequality for sections in the kernel of
Clifford contraction
(see \cite[Lemma~3.2]{WangZhu24}) gives
\begin{equation}\label{eq:penrose-refined-kato}
    \sum_{i=1}^n
    \bigl(\operatorname{Re}\langle\mathcal P_i\sigma,\sigma\rangle\bigr)^2
    \leq \frac{n-1}{n}|\sigma|^2\sum_{i=1}^n|\mathcal P_i\sigma|^2 .
\end{equation}

We now compare $\nabla\sigma$ with its Penrose part. Since
\[
    \sigma_i=\mathcal P_i\sigma-\frac1n c(e_i)\widetilde D\sigma,
\]
write
\[
    A_i=\operatorname{Re}\langle\mathcal P_i\sigma,\sigma\rangle,
    \qquad
    B_i=-\frac1n\operatorname{Re}\langle c(e_i)\widetilde D\sigma,\sigma\rangle .
\]
Then $\operatorname{Re}\langle\sigma_i,\sigma\rangle=A_i+B_i$. Therefore
\[
\begin{aligned}
    |\sigma|^2|\nabla|\sigma||^2
    &=\sum_{i=1}^n(A_i+B_i)^2  \\
    &=\sum_{i=1}^n A_i^2
      +2\sum_{i=1}^n A_iB_i
      +\sum_{i=1}^n B_i^2 .
\end{aligned}
\]
The $A_i$-term has the sharp coefficient. Indeed, by
\eqref{eq:penrose-refined-kato} and \eqref{eq:penrose-norm-identity},
\begin{equation}\label{eq:A-term-estimate}
    \sum_{i=1}^n A_i^2
    \leq \frac{n-1}{n}|\sigma|^2
    \left(|\nabla\sigma|^2-\frac1n|\widetilde D\sigma|^2\right)
    \leq \frac{n-1}{n}|\sigma|^2|\nabla\sigma|^2 .
\end{equation}
The $B_i$-terms are lower order. Since Clifford multiplication by a unit
vector is an isometry,
\[
    |B_i|\leq \frac1n |D\sigma|\,|\sigma|,
    \qquad
    \sum_{i=1}^n B_i^2\leq \frac1n |D\sigma|^2|\sigma|^2 .
\]
Finally, Cauchy--Schwarz and \eqref{eq:A-term-estimate} give
\[
\begin{aligned}
    \left|\sum_{i=1}^n A_iB_i\right|
    \leq
    \left(\sum_{i=1}^n A_i^2\right)^{1/2}
    \left(\sum_{i=1}^n B_i^2\right)^{1/2}  \leq
    \sqrt{\frac{n-1}{n^2}}|\sigma|^2|\widetilde D\sigma|\,|\nabla\sigma|.
\end{aligned}
\]
Combining the preceding estimates, we obtain, for a dimensional constant
$c_n>0$,
\[
    |\sigma|^2|\nabla|\sigma||^2
    \leq \frac{n-1}{n}|\sigma|^2|\nabla\sigma|^2
    +c_n|\sigma|^2|\widetilde D\sigma|^2
    +c_n|\sigma|^2|\widetilde D\sigma|\,|\nabla\sigma|.
\]
Dividing by $|\sigma|^2$ proves the
desired inequality at all points where $\sigma\neq0$. The inequality extends
to the zero set by continuity.
\end{proof}

The scalar curvature rigidity part relies on a quantitative unique
continuation theorem for low-energy spinors of the Dirac operator. This result
is the main analytic input in the rigidity argument, which prevents such spinors
from concentrating away from a prescribed orbit net.
\begin{theorem}\label{thm:uniqueContinuation}
Let $(M,g)$ be a closed Riemannian manifold with fundamental group
$\Gamma$, and let $(\widetilde M,\widetilde g)$ be its Riemannian universal
cover. Assume that $\widetilde M$ is spin, and let $\widetilde D$ be the
Dirac operator on $\widetilde M$. Fix $p\in M$ and a lift $q\in\widetilde M$
of $p$. Then, for every $r>0$, there exist constants $c_1,c_2>0$ such that
for every $\lambda>0$ and every spinor $\sigma\in L^2(S(T\widetilde M))$
lying in the spectral subspace of $\widetilde D^2$ with spectrum
$\leq\lambda$, one has
\[
    \int_{\widetilde M}|\sigma|^2
    \leq c_1\exp(c_2\lambda)
    \sum_{\gamma\in\Gamma}\int_{B_r(\gamma q)}|\sigma|^2 .
\]
\end{theorem}

The following estimate is the Dirac version of the low-energy control needed
below.

\begin{proposition}\label{prop:Lichnerowicz}
Assume the hypotheses of Theorem~\ref{thm:rigidity}. Then there
exists a constant $C>0$ with the following property. For every
$0<\varepsilon\leq 1$ and every
$\sigma\in H^2(\widetilde M,S(T\widetilde M))$ contained in the spectral
subspace of $\widetilde D$ corresponding to the interval
$[-\varepsilon,\varepsilon]$, one has
\begin{enumerate}
\item 
$\displaystyle\int_{\widetilde M}(|\widetilde D\sigma|^2+|\widetilde D^2\sigma^2|)\leq (\varepsilon^2+\varepsilon^4)\|\sigma\|^2;$
\item the refined Kato defect estimate
    \[
    0\leq \int_{\widetilde M}
    \left(|\nabla\sigma|^2-\frac{n}{n-1}|\nabla|\sigma||^2
    +c_n|\widetilde D\sigma|^2
    +c_n|\widetilde D\sigma||\nabla\sigma|\right)
    \leq C\varepsilon\|\sigma\|^2;
    \]

\item the bottom-spectrum defect estimate
    \[
    0\leq \int_{\widetilde M}
    \left(|\nabla|\sigma||^2-\frac{(n-1)^2}{4}|\sigma|^2\right)
    \leq C\varepsilon\|\sigma\|^2 .
    \]
\end{enumerate}
\end{proposition}
\begin{proof}
As shown in Theorem~\ref{thm: cocompact_rigidity}, we first observe that
$\Sc_g=-n(n-1)$. Since $\sigma$ belongs to the spectral subspace of
$\widetilde D$ corresponding to $[-\varepsilon,\varepsilon]$, the functional
calculus gives
\[
    \|\widetilde D\sigma\|\leq \varepsilon\|\sigma\|,\qquad\textup{and}\qquad\|\widetilde D^2\sigma\|\leq \varepsilon^2\|\sigma\|.
\]

Set
\[
\begin{split}
A(\sigma):=&\int_{\widetilde M}
    \left(|\nabla\sigma|^2
    -\frac{n}{n-1}|\nabla|\sigma||^2
    +c_n|\widetilde D\sigma|^2
    +c_n|\widetilde D\sigma||\nabla\sigma|\right),\\
B(\sigma):=&\int_{\widetilde M}
    \left(|\nabla|\sigma||^2
    -\frac{(n-1)^2}{4}|\sigma|^2\right).
\end{split}
\]
By the refined Kato inequality, $A(\sigma)\geq 0$. By the sharp bottom
spectrum assumption, $B(\sigma)\geq 0$.

By the Lichnerowicz formula
$$\widetilde D^2=\nabla^*\nabla+\frac{\Sc_{\widetilde g}}{4}=\nabla^*\nabla-\frac{n(n-1)}{4},$$
we have
\begin{equation}\label{eq:Lichnerowicz}
\begin{split}
\int_{\widetilde M} |\widetilde D\sigma|^2
=&\int_{\widetilde M}|\nabla\sigma|^2
    -\frac{n(n-1)}{4}\int_{\widetilde M}|\sigma|^2\\
=&\int_{\widetilde M}
    \left(|\nabla\sigma|^2
    -\frac{n}{n-1}|\nabla|\sigma||^2\right)\\
&+\frac{n}{n-1}\int_{\widetilde M}
    \left(|\nabla|\sigma||^2
    -\frac{(n-1)^2}{4}|\sigma|^2\right).
\end{split}
\end{equation}

We now add the nonnegative error terms from the refined Kato inequality to
both sides of \eqref{eq:Lichnerowicz}. This gives
\begin{equation}\label{eq:low-energy-defects}
A(\sigma)+\frac{n}{n-1}B(\sigma)=\int_{\widetilde M}
    \left(|\widetilde D\sigma|^2
    +c_n|\widetilde D\sigma|^2
    +c_n|\widetilde D\sigma||\nabla\sigma|\right).
\end{equation}

By the Lichnerowicz formula and
$\Sc_{\widetilde g}=-n(n-1)$, we have
\[
    \|\nabla\sigma\|^2
    =\|\widetilde D\sigma\|^2+\frac{n(n-1)}{4}\|\sigma\|^2
    \leq
    \left(\varepsilon^2+\frac{n(n-1)}{4}\right)\|\sigma\|^2 .
\]
Consequently, Cauchy--Schwarz yields
\begin{equation*}
\begin{split}
&\int_{\widetilde M}
    \left(|\widetilde D\sigma|^2+c_n|\widetilde D\sigma|^2
    +c_n|\widetilde D\sigma||\nabla\sigma|\right)\\
&\qquad \leq
    (1+c_n)\|\widetilde D\sigma\|^2
    +c_n\|\widetilde D\sigma\|\,\|\nabla\sigma\|\\
&\qquad \leq
    \left((1+c_n)\varepsilon^2
    +c_n\varepsilon
    \sqrt{\varepsilon^2+\frac{n(n-1)}{4}}\right)\|\sigma\|^2 .
\end{split}
\end{equation*}

Since $0<\varepsilon\leq 1$, the right-hand side is bounded by
$C\varepsilon\|\sigma\|^2$ for some dimensional constant $C>0$. Equation
\eqref{eq:low-energy-defects} and the preceding estimate give
\[
    A(\sigma)+\frac{n}{n-1}B(\sigma)
    \leq C\varepsilon\|\sigma\|^2 .
\]
Since both terms on the left are nonnegative, the two claimed estimates
follow.
\end{proof}

Proposition~\ref{prop:Lichnerowicz} provides the uniform estimates needed for
the equality case, but passing from these estimates to a geometrically useful
limit requires some care. The first issue is the noncompactness of the
universal cover. Since $\widetilde M$ is noncompact in general, there is no
global compactness theorem that would give strong convergence of a sequence of
spinors on all of $\widetilde M$.

The second issue is nonvanishing. A locally strongly convergent subsequence
may still have zero limit, because its $L^2$-mass may escape to infinity. We
therefore use the deck transformation group to translate the relevant mass
back into a fixed ball before taking the limit. A third issue is local
analytic in nature. The vanishing of the Kato defect will give a nonlinear
first-order equation \eqref{eq:Kato}. This non-linear relation does not directly create a  strictly positive local obstruction to the solvability of $\sigma$, which was utilized in \cite{WangZhu24}. Finally, the bottom-spectrum defect estimate (3) is established only in a global integral sense. It yields no direct geometric information for local integrals unless it is strictly paired with compactly supported local test functions. These difficulties will all be systematically addressed and resolved in Lemma \ref{lemma:limit} below.

\section{Proof of the rigidity theorem}\label{sec:proof-rigidity}
In this section, we prove Theorem~\ref{thm:rigidity}.

\subsection{Local limits of low-energy spinors}

We begin by constructing a nonzero local limit of low-energy spinors on the
universal cover.

\begin{lemma}\label{lemma:limit}
Assume the hypotheses of Theorem~\ref{thm:rigidity}. Let $p\in M$, let
$q\in\widetilde M$ be a lift of $p$, and choose $r>0$ such that
$3r$ is smaller than the injectivity radius of $(M,g)$. Then there exists a
nonzero smooth spinor $\sigma$ on $B_{2r}(q)$ satisfying
\begin{enumerate}[label=(\alph*)]
    \item \label{D=0} $\widetilde D\sigma=0$ on $B_{2r}(q)$;
    \item \label{Kato=} $\displaystyle|\nabla\sigma|^2
        =\frac{n}{n-1}|\nabla|\sigma||^2$, and
    \item \label{Delta=} $-\displaystyle\Delta|\sigma|
        =\frac{(n-1)^2}{4}|\sigma|$.
\end{enumerate}
\end{lemma}

\begin{proof}
Set $\Gamma=\pi_1(M)$. We may assume, without loss of generality, that $M$ is
spin. If $\widetilde M$ is spin but $M$ is not, the same argument applies
after replacing $\Gamma$ by the double cover $\hat\Gamma$ acting on
$\widetilde M$, which is the standard lift of the $\Gamma$-action to the spin
structure (see \cite[Theorem~3.5]{MR720934}).
\medskip

By Proposition~\ref{prop:zeroInSpec},
zero lies in the spectrum of
$\widetilde D$. Hence, for every $\varepsilon\in(0,1)$, the spectral subspace
of $\widetilde D$ corresponding to $[-\varepsilon,\varepsilon]$ is nontrivial.
Choose a spinor $\sigma^\varepsilon$ in this subspace and normalize it by
$\|\sigma^\varepsilon\|=1$.

We localize these spinors near the orbit of $q$. For each
$\gamma\in\Gamma$, pull back the restriction of $\sigma^\varepsilon$ to
$B_{3r}(\gamma q)$ by the deck transformation $\gamma^{-1}$, and write
\[
    \sigma^\varepsilon_\gamma
    :=\gamma^{-1}\left(
    \sigma^\varepsilon\big|_{B_{3r}(\gamma q)}\right).
\]
Thus $\sigma^\varepsilon_\gamma$ is a spinor on $B_{3r}(q)$. Since
$\sigma^\varepsilon$ lies in the spectral subspace of $\widetilde D$ with
spectrum contained in $[-\varepsilon,\varepsilon]$,
Theorem~\ref{thm:uniqueContinuation}, applied with radius $2r$, gives
\[
    1=\|\sigma^\varepsilon\|^2
    \leq c_1\exp(c_2\varepsilon^2)
    \sum_{\gamma\in\Gamma}
    \int_{B_{2r}(\gamma q)}|\sigma^\varepsilon|^2 .
\]
Because $0<\varepsilon<1$, it follows that there exists a constant
$C_1>0$, independent of $\varepsilon$, such that
\[
    \sum_{\gamma\in\Gamma}
    \int_{B_{2r}(q)}|\sigma_\gamma^\varepsilon|^2\geq C_1 .
\]

We next record the consequence of the sharp bottom spectrum assumption that
will be used to control the distributional equation for
$|\sigma^\varepsilon|$. Define
\begin{equation*}
    Q(\varphi_1,\varphi_2)
    :=\int_{\widetilde M}
    \left(\langle\nabla\varphi_1,\nabla\varphi_2\rangle
    -\frac{(n-1)^2}{4}\varphi_1\varphi_2\right)
\end{equation*}
first for $\varphi_1,\varphi_2\in C_c^\infty(\widetilde M)$, and then extend
it by closure to $H^1(\widetilde M)$. Since
$\lambda_1(\widetilde M,\widetilde g)=\frac{(n-1)^2}{4}$, we have
\[
    Q(\varphi,\varphi)\geq 0, \qquad\forall\varphi\in H^1(\widetilde M).
\]
Thus the Cauchy--Schwarz inequality for this non-negative bilinear form gives
\[
    |Q(\varphi_1,\varphi_2)|
    \leq \sqrt{Q(\varphi_1,\varphi_1)\cdot Q(\varphi_2,\varphi_2)}, \qquad\forall\varphi_1,\varphi_2\in H^1(\widetilde M) .
\]
Applying this to $|\sigma^\varepsilon|$ and using
Proposition~\ref{prop:Lichnerowicz}, we obtain, for every
$\varphi\in C_c^\infty(\widetilde M)$,
\begin{equation}\label{eq:form-small}
    |Q(|\sigma^\varepsilon|,\varphi)|
    \leq \sqrt{Q(|\sigma^\varepsilon|,|\sigma^\varepsilon|)\cdot Q(\varphi,\varphi)}\leq
    \sqrt{C\varepsilon}\cdot \sqrt{Q(\varphi,\varphi)} .
\end{equation}

\medskip

\emph{Claim.} Suppose that $\varepsilon>0$ satisfies
\[
    \varepsilon<
    \min\left\{1,\left(\frac{C_1}{4(2+\sqrt{C}+C)}\right)^6\right\}.
\]
Then there exists $\gamma^\varepsilon\in\Gamma$ such that
\begin{enumerate}[label=(\arabic*)]
    \item \label{cond:ne0}
        $\sigma^\varepsilon_{\gamma^\varepsilon}$ is not identically zero
        on $B_{2r}(q)$;
    \item \label{cond:3<=2} $\displaystyle
        \|\sigma_{\gamma^\varepsilon}^\varepsilon\|^2
        \leq\frac{2}{C_1}\int_{B_{2r}(q)}
        |\sigma_{\gamma^\varepsilon}^\varepsilon|^2$;
    \item \label{cond:D<=} $\displaystyle
        \int_{B_{3r}(q)}
        \left(|\widetilde D\sigma_{\gamma^\varepsilon}^\varepsilon|^2
        +|\widetilde D^2\sigma_{\gamma^\varepsilon}^\varepsilon|^2\right)
        \leq\varepsilon^{1/3}
        \|\sigma^\varepsilon_{\gamma^\varepsilon}\|^2$;
    \item \label{cond:Kato<=}$\displaystyle
        \int_{B_{3r}(q)}
        \left(|\nabla\sigma_{\gamma^\varepsilon}^\varepsilon|^2
        -\frac{n}{n-1}|\nabla|\sigma_{\gamma^\varepsilon}^\varepsilon||^2
        +c_n|\widetilde D\sigma_{\gamma^\varepsilon}^\varepsilon|^2
        +c_n|\widetilde D\sigma_{\gamma^\varepsilon}^\varepsilon|
        |\nabla\sigma_{\gamma^\varepsilon}^\varepsilon|\right)
        \leq \varepsilon^{1/3}
        \|\sigma^\varepsilon_{\gamma^\varepsilon}\|^2$;
    \item \label{cond:Detla<=}$\displaystyle
        |Q(|\sigma^\varepsilon_{\gamma^\varepsilon}|,\varphi)|
        \leq \varepsilon^{1/3}
        \|\sigma^\varepsilon_{\gamma^\varepsilon}\|\cdot\sqrt{Q(\varphi,\varphi)}$
        for every $\varphi\in C_c^\infty(B_{3r}(q))$.
\end{enumerate}

\medskip

We prove the claim. Let
\begin{equation*}
    \Gamma^\varepsilon_1
    =\left\{\gamma\in\Gamma:
    \|\sigma^\varepsilon_\gamma\|^2
    \leq \frac{2}{C_1}\int_{B_{2r}(q)}
    |\sigma^\varepsilon_\gamma|^2\right\}.
\end{equation*}
Since the balls $B_{3r}(\gamma q)$ are pairwise disjoint, the normalization
$\|\sigma^\varepsilon\|=1$ implies
\begin{equation*}
    \sum_{\gamma\in\Gamma}\|\sigma^\varepsilon_\gamma\|^2
    \leq 1.
\end{equation*}
By the definition of $\Gamma^\varepsilon_1$,
\begin{equation*}
    \sum_{\gamma\notin\Gamma^\varepsilon_1}
    \int_{B_{2r}(q)}|\sigma^\varepsilon_\gamma|^2
    <\frac{C_1}{2}
    \sum_{\gamma\notin\Gamma^\varepsilon_1}
    \|\sigma^\varepsilon_\gamma\|^2
    \leq \frac{C_1}{2}\sum_{\gamma\in\Gamma}\|\sigma^\varepsilon_\gamma\|^2\leq\frac{C_1}{2}.
\end{equation*}
Combining this with the lower bound
\[
    \sum_{\gamma\in\Gamma}
    \int_{B_{2r}(q)}|\sigma^\varepsilon_\gamma|^2\geq C_1
\]
gives
\begin{equation*}
    \sum_{\gamma\in\Gamma^\varepsilon_1}
    \int_{B_{2r}(q)}|\sigma^\varepsilon_\gamma|^2
    \geq C_1-\frac{C_1}{2}=\frac{C_1}{2}.
\end{equation*}
In particular, $\Gamma^\varepsilon_1$ is nonempty. Since the last sum consists
of nonnegative terms, we may choose a finite
subset
$\Gamma_2^\varepsilon\subset\Gamma_1^\varepsilon$ such that
$\sigma^\varepsilon_\gamma\ne 0$ for all
$\gamma\in\Gamma_2^\varepsilon$ and
\begin{equation*}
    \sum_{\gamma\in\Gamma_2^\varepsilon}
    \int_{B_{2r}(q)}|\sigma^\varepsilon_\gamma|^2
    \geq \frac{C_1}{4}.
\end{equation*}
Every element of $\Gamma^\varepsilon_2$ satisfies \ref{cond:ne0} and
\ref{cond:3<=2}.

\medskip

It remains to prove that at least one element of $\Gamma^\varepsilon_2$ also
satisfies \ref{cond:D<=}--\ref{cond:Detla<=}. Suppose, to the contrary, that
no element of $\Gamma_2^\varepsilon$ satisfies all three estimates. Then, for
each $\gamma\in\Gamma_2^\varepsilon$, at least one of the following
alternatives holds:
\begin{enumerate}[label=(\arabic*)', start=3]
    \item $\displaystyle
        \int_{B_{3r}(q)}
        \left(|\widetilde D\sigma_{\gamma}^\varepsilon|^2
        +|\widetilde D^2\sigma_{\gamma}^\varepsilon|^2\right)
        > \varepsilon^{1/3}\|\sigma^\varepsilon_{\gamma}\|^2,$
    \item $\displaystyle
        \int_{B_{3r}(q)}
        \left(|\nabla\sigma_{\gamma}^\varepsilon|^2
        -\frac{n}{n-1}|\nabla|\sigma_{\gamma}^\varepsilon||^2
        +c_n|\widetilde D\sigma_{\gamma}^\varepsilon|^2
        +c_n|\widetilde D\sigma_{\gamma}^\varepsilon|
        |\nabla\sigma_{\gamma}^\varepsilon|\right)
        >\varepsilon^{1/3}\|\sigma^\varepsilon_{\gamma}\|^2$,
    \item $\displaystyle
        |Q(|\sigma^\varepsilon_\gamma|,\varphi_\gamma)|
        >\varepsilon^{1/3}
        \|\sigma^\varepsilon_{\gamma}\|\cdot \sqrt{Q(\varphi_\gamma,\varphi_\gamma)}$
        for some local test function
        $\varphi_\gamma\in C_c^\infty(B_{3r}(q))$.
\end{enumerate}

Let $\Gamma^\varepsilon_{2,3'}$, $\Gamma^\varepsilon_{2,4'}$, and
$\Gamma^\varepsilon_{2,5'}$ denote the subsets of $\Gamma_2^\varepsilon$
on which $(3)'$, $(4)'$, and $(5)'$ hold, respectively. These three subsets
cover $\Gamma_2^\varepsilon$, since every $\gamma\in\Gamma_2^\varepsilon$
satisfies at least one of the alternatives $(3)'$, $(4)'$, and $(5)'$.

For each element of $\Gamma^\varepsilon_{2,5'}$, choose one test function
$\varphi_\gamma$ satisfying $(5)'$, and normalize it and, if necessary, change
its sign so that
$Q(\varphi_\gamma,\varphi_\gamma)=1$ and
$Q(|\sigma^\varepsilon_\gamma|,\varphi_\gamma)>0$. Define the global test
function
\begin{equation*}
    \varphi
    =\sum_{\gamma\in\Gamma_{2,5'}^\varepsilon}
    \gamma(\varphi_\gamma)\cdot\|\sigma^\varepsilon_\gamma\|
    \in C_c^\infty(\widetilde M).
\end{equation*}
The supports of the summands are pairwise disjoint. Moreover, on
$B_{3r}(\gamma q)$ the function $|\sigma^\varepsilon|$ pulls back by
$\gamma^{-1}$ to $|\sigma^\varepsilon_\gamma|$, and the form $Q$ is local
and invariant under deck transformations. Thus there are no cross terms in
the $Q$-pairings. Since $Q(\varphi_\gamma,\varphi_\gamma)=1$, we have
\[
    Q(\varphi,\varphi)
    =\sum_{\gamma\in\Gamma_{2,5'}^\varepsilon}
    \|\sigma^\varepsilon_\gamma\|^2\leq 1 .
\]
Hence, using \eqref{eq:form-small}, we get
\begin{equation*}
    \begin{split}
        \sum_{\gamma\in\Gamma_{2,5'}^\varepsilon}
        \|\sigma^\varepsilon_\gamma\|
        Q(|\sigma^\varepsilon_\gamma|,\varphi_\gamma)
        &=Q(|\sigma^\varepsilon|,\varphi)\\
        &\leq \sqrt{C\varepsilon}\,\sqrt{Q(\varphi,\varphi)}    =\sqrt{C\varepsilon}
        \sqrt{\sum_{\gamma\in\Gamma_{2,5'}^\varepsilon}
        \|\sigma^\varepsilon_\gamma\|^2}
        \leq \sqrt{C\varepsilon}.
    \end{split}
\end{equation*}

Set
\begin{equation*}
    \begin{split}
        I\coloneqq&
        \sum_{\gamma\in\Gamma_2^\varepsilon}\int_{B_{3r}(q)}
        \left(|\widetilde D\sigma_{\gamma}^\varepsilon|^2
        +|\widetilde D^2\sigma_{\gamma}^\varepsilon|^2\right)
        +\sum_{\gamma\in\Gamma_{2,5'}^\varepsilon}
        \|\sigma_\gamma^\varepsilon\|
        Q(|\sigma_\gamma^\varepsilon|,\varphi_\gamma)\\
        &+\sum_{\gamma\in\Gamma^\varepsilon_2}\int_{B_{3r}(q)}
        \left(|\nabla\sigma_{\gamma}^\varepsilon|^2
        -\frac{n}{n-1}|\nabla|\sigma_{\gamma}^\varepsilon||^2
        +c_n|\widetilde D\sigma_{\gamma}^\varepsilon|^2
        +c_n|\widetilde D\sigma_{\gamma}^\varepsilon|
        |\nabla\sigma_{\gamma}^\varepsilon|\right).
    \end{split}
\end{equation*}
Each summand above is nonnegative.
Since
\[
    \Gamma_2^\varepsilon
    =\Gamma^\varepsilon_{2,3'}\cup
    \Gamma^\varepsilon_{2,4'}\cup
    \Gamma^\varepsilon_{2,5'},
\]
each $\gamma\in \Gamma_{2,i'}^\varepsilon$ contributes more than
$\varepsilon^{1/3}\|\sigma_\gamma^\varepsilon\|^2$ to the definition of
$I$ through the corresponding bad inequality $(i)'$. Hence
\begin{equation*}
    I>\varepsilon^{1/3}
    \sum_{\gamma\in\Gamma_2^\varepsilon}
    \|\sigma_\gamma^\varepsilon\|^2 \geq \varepsilon^{1/3}\cdot\frac{C_1}{4}.
\end{equation*}
On the other hand, Proposition~\ref{prop:Lichnerowicz}, together with the
estimate above for the $(5)'$-terms, gives
\begin{equation*}
    I\leq \varepsilon^2+\varepsilon^4+\sqrt{C\varepsilon}
    +C\varepsilon
    \leq (2+\sqrt{C}+C)\varepsilon^{1/2}.
\end{equation*}
This contradicts our choice of $\varepsilon$. The claim follows.

\medskip
Finally, choose a sequence $\varepsilon_j\downarrow 0$ satisfying the
smallness condition above, and let
$\gamma_j=\gamma^{\varepsilon_j}$ be given by the claim. We normalize the
corresponding localized spinors by setting
\begin{equation*}
    \overline{\sigma}^{\,j}
    =\frac{\sigma^{\varepsilon_j}_{\gamma_j}}
    {\|\sigma^{\varepsilon_j}_{\gamma_j}\|_{L^2(B_{2r}(q))}} .
\end{equation*}
Then
\begin{equation*}
    \int_{B_{2r}(q)}|\overline{\sigma}^{\,j}|^2=1.
\end{equation*}
Moreover, condition~\ref{cond:3<=2} gives the uniform local $L^2$ bound
\begin{equation*}
    \int_{B_{3r}(q)}|\overline{\sigma}^{\,j}|^2
    \leq \frac{2}{C_1}.
\end{equation*}
After this normalization, condition~\ref{cond:D<=} implies
\begin{equation*}
    \int_{B_{3r}(q)}
    \left(|\widetilde D\overline{\sigma}^{\,j}|^2
    +|\widetilde D^2\overline{\sigma}^{\,j}|^2\right)
    \leq \varepsilon_j^{1/3}.
\end{equation*}
Similarly, conditions~\ref{cond:Kato<=} and
\ref{cond:Detla<=} become
\begin{equation*}
    \int_{B_{3r}(q)}
    \left(|\nabla\overline{\sigma}^{\,j}|^2
    -\frac{n}{n-1}|\nabla|\overline{\sigma}^{\,j}||^2
    +c_n|\widetilde D\overline{\sigma}^{\,j}|^2
    +c_n|\widetilde D\overline{\sigma}^{\,j}|
    |\nabla\overline{\sigma}^{\,j}|\right)
    \leq \varepsilon_j^{1/3}
\end{equation*}
and, for every $\varphi\in C_c^\infty(B_{3r}(q))$,
\begin{equation*}
    |Q(|\overline{\sigma}^{\,j}|,\varphi)|
    \leq \varepsilon_j^{1/3}\sqrt{Q(\varphi,\varphi)}.
\end{equation*}

Local ellipticity for
$\widetilde D$, together with the preceding $L^2$-bound and the
$\widetilde D,\widetilde D^2$ bounds
on $B_{3r}(q)$, gives a uniform
$H^2(B_{2r}(q),S(T\widetilde M))$-bound for
$\overline{\sigma}^{\,j}$. Passing to a subsequence, we may assume that
\[
    \overline{\sigma}^{\,j}\to \sigma
    \quad\text{in } H^1(B_{2r}(q),S(T\widetilde M)).
\]
The convergence and normalization give
$$\int_{B_{2r}(q)}|\sigma|^2=1,$$
so $\sigma$ is nonzero. Passing to the limit using \ref{cond:D<=}, we obtain
$\widetilde D\sigma=0$. The equality case in the refined Kato inequality
follows by passing to the
limit in \ref{cond:Kato<=} and
Proposition~\ref{prop:refinedKato}. Namely,
\[
    |\nabla\sigma|^2=\frac{n}{n-1}|\nabla|\sigma||^2
    \quad\text{on } B_{2r}(q).
\]
Finally, for every $\varphi\in C_c^\infty(B_{2r}(q))$, the estimate for $Q$ gives
\begin{equation*}
    \int_{B_{2r}(q)}
    \left(\langle\nabla|\sigma|,\nabla\varphi\rangle
    -\frac{(n-1)^2}{4}|\sigma|\varphi\right)=0.
\end{equation*}
Thus $|\sigma|$ is a weak solution of
$-\Delta|\sigma|=\frac{(n-1)^2}{4}|\sigma|$ on $B_{2r}(q)$. Standard elliptic
regularity then implies that $\sigma$ is smooth, and the three asserted
properties follow.
\end{proof}

\subsection{Proof of Theorem~\ref{thm:rigidity}}
Having secured the existence of the strongly constrained local spinor, we now
translate the spinorial geometry into pure Riemannian geometry. For the
reader's convenience, we recall what remains to be proved: under the
hypotheses of Theorem~\ref{thm:rigidity}, the universal cover
$(\widetilde M,\widetilde g)$ must be isometric to hyperbolic space
$\mathbb H^n$. Since Theorem~\ref{thm: cocompact_rigidity} has already given
$\Sc_g\equiv -n(n-1)$, it is enough to prove that $g$ is Einstein,
\[
    \Ric_g=-(n-1)g.
\]
Indeed, the desired hyperbolic rigidity then follows from X. Wang's sharp
eigenvalue rigidity theorem (see \cite[Theorem~1.4]{Wang_eigenvalue}).
\begin{proof}[Proof of Theorem~\ref{thm:rigidity}]
Let us first assume that $n\geq 3$ since the scalar curvature is twice of the Gaussian curvature as $n=2$. 

Fix any point $p\in M$, and choose a lift $q\in\widetilde M$.
It suffices to prove the curvature identity purely at the point $q$.
Let $\sigma$ be the nontrivial limiting spinor on $B_{2r}(q)$ obtained in
Lemma~\ref{lemma:limit}. Since $|\sigma|\geq 0$ weakly solves the Laplacian
equation~\ref{Delta=}, the Harnack inequality implies that $|\sigma|>0$ in the interior ball $B_r(q)$.

Set $V=\nabla(\log|\sigma|)$. By Lemma~\ref{lemma:limit}, $\sigma$ realizes
equality in the refined Kato inequality. The equality case gives the
first-order nonlinear equation
\begin{equation}\label{eq:Kato}
	\nabla\sigma=\frac{1}{n-1}\sum_{i=1}^n e_i\otimes c(e_i)c(V)\sigma+\frac{n}{n-1}V\otimes \sigma,
\end{equation}
where $\{e_i\}$ is a local orthonormal frame near $q$ with $\nabla e_i(q)=0$.

Let $X$ be an arbitrary vector field near $q$ that  satisfies
$\nabla X(q)=0$.
By the spinorial curvature identity
(see \cite[Corollary 2.8]{spinorialapproach}), we have
\begin{equation*}
    [D,\nabla_X]=\sum_{i=1}^n c(e_i)[\nabla_{e_i},\nabla_X]=\frac{1}{2}c(\Ric_g(X)).
\end{equation*}
Since $D\sigma=0$ by Lemma~\ref{lemma:limit}, we obtain
\begin{equation*}
\begin{split}
		\frac{1}{2} c(\Ric_g(X))\sigma &= D\left(\frac{1}{n-1} c(X)c(V)\sigma+\frac{n}{n-1}\langle V,X\rangle\sigma\right)\\
		&= \frac{1}{n-1}\underbrace{\sum_{i=1}^n c(e_i)c(X)c(\nabla_{e_i} V)\sigma}_{I_1}+\frac{1}{n-1}\underbrace{\sum_{i=1}^n c(e_i)c(X)c(V)\nabla_{e_i}\sigma}_{I_2}\\
		&\quad+\frac{n}{n-1}\underbrace{\sum_{i=1}^n e_i(\langle V,X\rangle)c(e_i)\sigma}_{I_3}.
\end{split}
\end{equation*}

We now simplify the three terms $I_1$, $I_2$, and $I_3$ one by one below.

\begin{enumerate}
\item For \(I_1\), since \(V\) is a gradient field, the Hessian \(\nabla V\)
is symmetric. Using the Clifford relation
\[
c(e_i)c(X)+c(X)c(e_i)=-2\langle e_i,X\rangle,
\]
we obtain
\[
\begin{aligned}
I_1
&=\sum_{i=1}^n c(e_i)c(X)c(\nabla_{e_i}V)\sigma  \\
&=-\sum_{i=1}^n c(X)c(e_i)c(\nabla_{e_i}V)\sigma
  -2\sum_{i=1}^n \langle X,e_i\rangle c(\nabla_{e_i}V)\sigma  \\
&=-c(X)\sum_{i=1}^n c(e_i)c(\nabla_{e_i}V)\sigma
  -2c(\nabla_XV)\sigma .
\end{aligned}
\]
To simplify the first term, write
\[
\nabla_{e_i}V=\sum_{j=1}^n h_{ij}e_j,
\qquad h_{ij}=h_{ji}.
\]
Then
\[
\begin{aligned}
\sum_{i=1}^n c(e_i)c(\nabla_{e_i}V)
&=\sum_{i,j=1}^n h_{ij}c(e_i)c(e_j)  \\
&=\sum_{i=1}^n h_{ii}c(e_i)^2
  +\sum_{i<j}h_{ij}\bigl(c(e_i)c(e_j)+c(e_j)c(e_i)\bigr) \\
&=-\sum_{i=1}^n h_{ii}
=-\divv(V).
\end{aligned}
\]
Therefore
\[
I_1
=
\divv(V)c(X)\sigma-2c(\nabla_XV)\sigma .
\]

\item For \(I_2\), the Clifford relation gives
\[
\begin{aligned}
c(e_i)c(X)c(V)
&=-c(X)c(e_i)c(V)-2\langle X,e_i\rangle c(V)  \\
&=c(X)c(V)c(e_i)
  +2\langle e_i,V\rangle c(X)
  -2\langle X,e_i\rangle c(V).
\end{aligned}
\]
Substituting this into \(I_2\), we get
\[
\begin{aligned}
I_2
&=\sum_{i=1}^n c(e_i)c(X)c(V)\nabla_{e_i}\sigma  \\
&=\sum_{i=1}^n c(X)c(V)c(e_i)\nabla_{e_i}\sigma  \\
&\quad
  +\sum_{i=1}^n
  \left(
  2\langle e_i,V\rangle c(X)
  -2\langle X,e_i\rangle c(V)
  \right)\nabla_{e_i}\sigma .
\end{aligned}
\]
The first sum vanishes because
\[
\sum_{i=1}^n c(e_i)\nabla_{e_i}\sigma=D\sigma=0.
\]
Hence
\[
\begin{aligned}
I_2
&=
2c(X)\sum_{i=1}^n \langle e_i,V\rangle \nabla_{e_i}\sigma
-2c(V)\sum_{i=1}^n \langle X,e_i\rangle \nabla_{e_i}\sigma  \\
&=2c(X)\nabla_V\sigma-2c(V)\nabla_X\sigma .
\end{aligned}
\]
We now use \eqref{eq:Kato}. First,
\[
\nabla_V\sigma
=
\frac{1}{n-1}c(V)c(V)\sigma
+\frac{n}{n-1}|V|^2\sigma
=
|V|^2\sigma,
\]
because \(c(V)^2=-|V|^2\). Also,
\[
\nabla_X\sigma
=
\frac{1}{n-1}c(X)c(V)\sigma
+\frac{n}{n-1}\langle V,X\rangle\sigma.
\]
Therefore
\[
\begin{aligned}
I_2
&=2|V|^2c(X)\sigma
  -2c(V)\left(
  \frac{1}{n-1}c(X)c(V)\sigma
  +\frac{n}{n-1}\langle V,X\rangle\sigma
  \right).
\end{aligned}
\]
Using
\[
c(V)c(X)c(V)=|V|^2c(X)-2\langle V,X\rangle c(V),
\]
we obtain
\[
\begin{aligned}
I_2
&=2|V|^2c(X)\sigma
  -\frac{2}{n-1}|V|^2c(X)\sigma
  +\frac{4}{n-1}\langle V,X\rangle c(V)\sigma  \\
&\quad
  -\frac{2n}{n-1}\langle V,X\rangle c(V)\sigma  \\
&=\frac{2(n-2)}{n-1}|V|^2c(X)\sigma
  -\frac{2(n-2)}{n-1}\langle V,X\rangle c(V)\sigma .
\end{aligned}
\]

\item For $I_3$, the symmetry of $\nabla V$ gives
\begin{equation*}
\begin{split}
    I_3
    = \sum_{i=1}^n\langle\nabla_{e_i} V,X\rangle c(e_i)\sigma
    = \sum_{i=1}^n\langle\nabla_X V,e_i\rangle c(e_i)\sigma
     = c(\nabla_X V)\sigma .
\end{split}
\end{equation*}

\end{enumerate}

Combining the three identities, we obtain
\[
\begin{aligned}
\frac{1}{2}c(\Ric_g(X))\sigma
&=
\frac{1}{n-1}I_1+\frac{1}{n-1}I_2+\frac{n}{n-1}I_3  \\
&=
\frac{1}{n-1}
\left(\divv(V)c(X)\sigma-2c(\nabla_XV)\sigma\right)  \\
&\quad
+\frac{2(n-2)}{(n-1)^2}
\left(
|V|^2c(X)\sigma-\langle V,X\rangle c(V)\sigma
\right)
+\frac{n}{n-1}c(\nabla_XV)\sigma  \\
&=
\frac{1}{n-1}\divv(V)c(X)\sigma
+\frac{n-2}{n-1}c(\nabla_XV)\sigma  \\
&\quad
+\frac{2(n-2)}{(n-1)^2}
\left(
|V|^2c(X)\sigma-\langle V,X\rangle c(V)\sigma
\right).
\end{aligned}
\]
Equivalently,
\[
c(Y)\sigma=0,
\]
where
\[
Y
=
-\frac{1}{2}\Ric_g(X)
+\frac{n-2}{n-1}\nabla_XV
+\frac{1}{n-1}\divv(V)X
+\frac{2(n-2)}{(n-1)^2}
\left(
|V|^2X-\langle V,X\rangle V
\right).
\]

Taking the norm of \(c(Y)\sigma=0\), and using
\[
|c(Y)\sigma|=|Y|\,|\sigma|,
\]
we conclude that \(|Y|\,|\sigma|=0\). Since \(\sigma\neq 0\) everywhere on
\(B_r(q)\), it follows that \(Y=0\). Thus
\begin{equation}\label{eq:V-eq1}
\frac{1}{2}\Ric_g(X)
=
\frac{n-2}{n-1}\nabla_XV
+\frac{1}{n-1}\divv(V)X
+\frac{2(n-2)}{(n-1)^2}
\left(|V|^2X-\langle V,X\rangle V\right).
\end{equation}

Taking the trace in \(X\) gives
\[
\begin{aligned}
\frac{1}{2}\Sc_g
&=
\frac{n-2}{n-1}\divv(V)
+\frac{n}{n-1}\divv(V)
+\frac{2(n-2)}{(n-1)^2}
\left(n|V|^2-|V|^2\right)  \\
&=
2\divv(V)
+\frac{2(n-2)}{n-1}|V|^2 .
\end{aligned}
\]
Since \(\Sc_g=-n(n-1)\), this yields
\[
\divv(V)
=
-\frac{n(n-1)}{4}
-\frac{n-2}{n-1}|V|^2 .
\]

On the other hand, because \(V=\nabla(\log|\sigma|)\), we have
\[
\begin{aligned}
\divv(V)
&=\Delta(\log|\sigma|)  \\
&=\frac{\Delta|\sigma|}{|\sigma|}
-\frac{|\nabla|\sigma||^2}{|\sigma|^2}  \\
&=\frac{\Delta|\sigma|}{|\sigma|}-|V|^2 .
\end{aligned}
\]
Using \eqref{Delta=} from Lemma~\ref{lemma:limit}, namely
\[
-\Delta|\sigma|=\frac{(n-1)^2}{4}|\sigma|,
\]
we obtain
\[
\divv(V)
=
-\frac{(n-1)^2}{4}
-|V|^2 .
\]
Comparing the two formulas for \(\divv(V)\), we get
\[
|V|^2=\frac{(n-1)^2}{4},
\qquad
\divv(V)=-\frac{(n-1)^2}{2}.
\]

Substituting these identities back into \eqref{eq:V-eq1}, we obtain
\[
\begin{aligned}
\frac{1}{2}\Ric_g(X)
&=
\frac{n-2}{n-1}\nabla_XV
+\frac{1}{n-1}\left(-\frac{(n-1)^2}{2}\right)X  \\
&\quad
+\frac{2(n-2)}{(n-1)^2}
\left(
\frac{(n-1)^2}{4}X-\langle V,X\rangle V
\right)  \\
&=
\frac{n-2}{n-1}\nabla_XV
-\frac{n-1}{2}X
+\frac{n-2}{2}X
-\frac{2(n-2)}{(n-1)^2}\langle V,X\rangle V  \\
&=
\frac{n-2}{n-1}\nabla_XV
-\frac{1}{2}X
-\frac{2(n-2)}{(n-1)^2}\langle V,X\rangle V .
\end{aligned}
\]
Thus
\begin{equation}\label{eq:V-eq2}
\frac{1}{2}\Ric_g(X)
=
\frac{n-2}{n-1}\nabla_XV
-\frac{1}{2}X
-\frac{2(n-2)}{(n-1)^2}\langle V,X\rangle V .
\end{equation}

\medskip
Define the traceless Ricci tensor by
\[
E(X)=\Ric_g(X)+(n-1)X.
\]
Then \eqref{eq:V-eq2} can be rewritten as
\begin{equation}\label{eq:E-eq}
\nabla_XV
=
\frac{n-1}{2(n-2)}E(X)
-\frac{n-1}{2}X
+\frac{2}{n-1}\langle V,X\rangle V .
\end{equation}

Since \(|V|^2\) is constant, we have
\[
\langle \nabla_XV,V\rangle=0.
\]
Taking the inner product of \eqref{eq:E-eq} with \(V\), and using
\[
|V|^2=\frac{(n-1)^2}{4},
\]
gives
\[
\begin{aligned}
0
&=
\frac{n-1}{2(n-2)}\langle E(X),V\rangle
-\frac{n-1}{2}\langle X,V\rangle
+\frac{2}{n-1}\langle V,X\rangle |V|^2  \\
&=
\frac{n-1}{2(n-2)}\langle E(X),V\rangle .
\end{aligned}
\]
Thus \(E(V)=0\), equivalently
\[
\Ric_g(V)=-(n-1)V.
\]

We next apply the Bochner formula to the function \(\log|\sigma|\), whose
gradient is \(V\):
\[
\frac{1}{2}\Delta |V|^2
=
|\nabla V|^2+\langle \Ric_g(V),V\rangle
+\langle V,\nabla(\divv V)\rangle .
\]
Since both \(|V|^2\) and \(\divv(V)\) are constant, this becomes
\begin{equation}\label{eq:|nablaV|^2}
|\nabla V|^2
=
(n-1)|V|^2
=
\frac{(n-1)^3}{4}.
\end{equation}

We now compute the same quantity from \eqref{eq:E-eq}. First, taking \(X=V\)
and using \(E(V)=0\), we get
\[
\nabla_VV
=
-\frac{n-1}{2}V
+\frac{2}{n-1}|V|^2V
=0.
\]
For any \(X\perp V\), equation \eqref{eq:E-eq} reduces to
\[
\nabla_XV
=
\frac{n-1}{2(n-2)}E(X)
-\frac{n-1}{2}X.
\]
Choose an orthonormal frame with \(e_1=V/|V|\) and
\(e_i\perp V\) for \(i=2,\ldots,n\). Since \(E(V)=0\), we have
\[
|E|^2=\sum_{\alpha=2}^n |E(e_i)|^2.
\]
Therefore
\[
\begin{aligned}
|\nabla V|^2
&=
\sum_{i=2}^n
\left|
\frac{n-1}{2(n-2)}E(e_i)
-\frac{n-1}{2}e_i
\right|^2  \\
&=
\frac{(n-1)^2}{4(n-2)^2}|E|^2
+\frac{(n-1)^3}{4}
-\frac{(n-1)^2}{2(n-2)}
\sum_{i=2}^n \langle E(e_i),e_i\rangle .
\end{aligned}
\]
Since \(E\) is traceless and \(E(V)=0\), the last sum is zero. Hence
\[
|\nabla V|^2
=
\frac{(n-1)^2}{4(n-2)^2}|E|^2
+\frac{(n-1)^3}{4}.
\]
Comparing this with \eqref{eq:|nablaV|^2}, we obtain
\[
|E|^2=0.
\]
Thus \(E\) vanishes at the point under consideration. Since the point \(p\) was
arbitrary, \(E\equiv 0\) on \(M\), and therefore
\[
\Ric_g=-(n-1)g.
\]
We now apply X. Wang's rigidity theorem in the Ricci curvature setting. Since
the universal cover satisfies
\[
\Ric_{\widetilde g}=-(n-1)\widetilde g
\qquad\text{and}\qquad
\lambda_1(\widetilde M,\widetilde g)=\frac{(n-1)^2}{4},
\]
X. Wang's characterization theorem implies that
\[
(\widetilde M,\widetilde g)\cong \mathbb H^n
\]
isometrically (see \cite[Theorem~1.4]{Wang_eigenvalue}). Therefore \((M,g)\)
is a closed hyperbolic manifold. This proves Theorem~\ref{thm:rigidity}.
\end{proof}
\bibliographystyle{amsplain}
\bibliography{ref}

\end{document}